\documentclass[a4paper,10pt]{article}

\usepackage[T1]{fontenc}
\usepackage{enumerate}
\usepackage{amsmath,amssymb,amsthm}
\usepackage{wasysym}
\usepackage{graphicx}
\usepackage{bbm}
\usepackage{dsfont}
\usepackage{todonotes} 
\usepackage{cite}      
\usepackage{relsize}
\usepackage{authblk}

\usepackage[hidelinks,bookmarksnumbered]{hyperref}
\hypersetup{pdfstartview=FitH}
\usepackage{tikz}
\usetikzlibrary{calc}
\usetikzlibrary{shapes}
\usetikzlibrary{intersections}

  \newcounter{constant}

\def\arraypar#1{\parbox[c]{\textwidth - 2cm}{\centering #1}}

\usepackage{environ}
\NewEnviron{display}{
  \begin{equation}\begin{array}{c}
    \arraypar{\BODY}
  \end{array}\end{equation}
}

\newcommand{\I}{\mathds{1}}

\newcommand{\dist}{\ensuremath{\mathop{\mathrm{dist}}\nolimits}}

\newcommand{\Cov}{\mathop{\mathrm{Cov}}\nolimits}

\newcommand{\norm}[1]{\left\lVert#1\right\rVert}

\newcommand{\super}[1]{\ensuremath{^{\textnormal{\scriptsize#1}}}}

\newcommand{\pcb}{p_c\super{bond}}
\newcommand{\pcs}{p_c\super{site}}
\newcommand{\dual}[1]{\smash{{(#1)}^{\ast}}}
\newcommand{\PPs}{\mathbb{P}\rule{0pt}{8pt}\super{\,s}}
\newcommand{\PPb}{\mathbb{P}\rule{0pt}{8pt}\super{\,b}}

\newcommand{\cC}{\ensuremath{\mathcal{C}}}

\newcommand{\cL}{\ensuremath{\mathcal{L}}}

\newcommand{\cN}{\ensuremath{\mathcal{N}}}
\newcommand{\cO}{\ensuremath{\mathcal{O}}}

\newcommand{\EE}{\ensuremath{\mathbb{E}}}

\newcommand{\LL}{\ensuremath{\mathbb{L}}}
\newcommand{\NN}{\ensuremath{\mathbb{N}}}
\newcommand{\PP}{\ensuremath{\mathbb{P}}}

\newcommand{\RR}{\ensuremath{\mathbb{R}}}
\newcommand{\ZZ}{\ensuremath{\mathbb{Z}}}

\theoremstyle{plain}
\newtheorem{prop}{Proposition}
\newtheorem{teo}{Theorem}
\newtheorem{lema}{Lemma}

\theoremstyle{definition}

\theoremstyle{remark}
\newtheorem{remark}{Remark}

\title{A note on the phase transition for independent alignment percolation}
\author[1]{Marcelo Hil\'ario}
\author[2]{Daniel Ungaretti}
\affil[1,2]{Universidade Federal de Minas Gerais}

\begin{document}
\maketitle

\begin{abstract}
    We study the independent alignment percolation model on
    $\mathbb{Z}^d$ introduced by Beaton, Grimmett and Holmes [arXiv:1908.07203]. 
    It is a model
    for random intersecting line segments defined as follows. First the sites
    of $\mathbb{Z}^d$ are independently declared occupied with probability $p$
    and vacant otherwise. Conditional on the configuration of occupied
    vertices, consider the set of all line segments that are parallel to the
    coordinate axis, whose extremes are occupied vertices and that do not
    traverse any other occupied vertex. Declare independently the segments on
    this set open with probability $\lambda$ and closed otherwise. All the
    edges that lie on open segments are also declared open giving rise to a
    bond percolation model in $\mathbb{Z}^d$. We show that for any $d \geq 2$
    and $p \in (0,1]$ the critical value for $\lambda$ satisfies
    $\lambda_c(p)<1$ completing the proof that the phase transition is
    non-trivial over the whole interval $(0,1]$.
    We also show that the critical curve $p \mapsto \lambda_c(p)$ is
    continuous at $p=1$, answering a question posed by the authors in
    [arXiv:1908.07203].
\end{abstract}

%
%



\section{Definition of the model}
\label{sec:definition_of_the_model}

Alignment percolation has been recently introduced by Beaton, Grimmett and
Holmes \cite{beaton2019alignment} as a model of random line segments in the
hypercubic lattice $\LL^d = (\ZZ^d, \EE^d)$ with $d \geq 2$. There, the
authors define two versions of the model, referred to as the `one-choice model'
and the `independent model' as we revisit below.

Fix $d \geq 2$ and let $\Omega := \{0,1\}^{\ZZ^d}$ be the set of configurations
$\omega$ on $\ZZ^d$.  Given $\omega = \{ \omega(v) ; v\in \mathbb{Z}^d\} \in
\Omega$ a site $v \in \ZZ^d$ is said \textit{occupied} when $\omega(v) = 1$ and
\textit{vacant} when $\omega(v) = 0$.
For any parameter $p \in (0, 1]$ the independent Bernoulli site percolation
model is the measure $\PPs_p$ on $\Omega$ under which $\{\omega(v);
v\in \mathbb{Z}^d\}$ is a family of independent Bernoulli random variables with
mean $p$. In other words, any site $v \in \ZZ^d$ independently is occupied or
vacant with probability $p$ and $1-p$, respectively.

Denote $\eta(\omega) = \{v \in \ZZ^d;\; \omega(v) = 1\}$. 
For a fixed configuration $\omega$, we say that a pair of sites
$v_1, v_2 \in \ZZ^d$ is \textit{feasible} when $v_1$ and $v_2$ differ only at
a single coordinate and both $v_1$ and $v_2$ belong to $\eta$ but no other
site in the line segment $v_1 v_2$ that connects $v_1$ and $v_2$ belongs to
$\eta$. We denote the set of feasible pairs by $F(\eta)$ and define a
random graph $G = G(\omega)$ whose vertex set is the set of occupied vertices in
$\omega$ and whose edges are the feasible pairs; that is,
$G(\omega) = (\eta, F(\eta))$.

In \cite{beaton2019alignment} two bond percolation models on $G$ are
defined by specifying measures on $\{0, 1\}^{F(\eta)}$:
\begin{description}
\item[One-choice model.] For every site $v \in \eta$ choose a feasible pair
    $f_v = vu \in F(\eta)$ uniformly among all $2d$ available possibilities. 
    Declare $e = uv \in F(\eta)$ open if either $e = f_u$ or $e = f_v$.
    Otherwise, declare it closed.

\item[Independent model.] Fix $\lambda \in [0,1]$ and declare each edge
    $e \in F(\eta)$ to be open independently with probability $\lambda$.
\end{description}

Both versions described above can be regarded as dependent bond percolation
models on $(\ZZ^d, \EE^{d})$ by declaring an edge $e \in \EE^d$ open if and
only if the unique edge of $F(\eta)$ that contains $e$ is open.  This gives
rise to a random element $\sigma$ in $\Sigma := \{0, 1\}^{\smash{\EE^{d}}}$
which in turn induces a random subgraph of $\mathbb{Z}^d$. 
We are mainly interested in studying connectivity properties of this random graph.

Here we focus in the independent version of the model. In this context, let
$P^{\omega}_{\lambda}$ denote the measure on $\Sigma$ induced by
independent Bernoulli percolation on $\{0,1\}^{F(\eta)}$ with parameter $\lambda$.
The distribution of the pair $\xi :=(\omega, \sigma)$ is
the probability measure $\mathbb{P}_{p,\lambda,d}$ on
$\Xi := \Omega \times \Sigma$ satisfying
\begin{equation*}
\mathbb{P}_{p, \lambda, d} (A \times B)
    := \int_{A} P^{\omega}_{\lambda}(B) \,\mathrm{d} \PPs_p(\omega)
    \quad \text{ $A \subset \Omega$, $B \subset \Sigma$ measurable},
\end{equation*}
where $\Omega$ and $\Sigma$ are endowed with their respective cylinder
$\sigma$-fields. We will often drop the dependency in $d$ from the notation
and write simply $\mathbb{P}_{p,\lambda}$.
In what follows we may abuse notation and write simply $\xi(v)$ and
$\xi(e)$ in order to refer to the state of a site $v$ and an edge $e$,
that is, $\omega(v)$ and $\sigma(e)$, respectively.

The set of \textit{open edges} of $\LL^d$ is denoted by $\cO = \{e \in \EE^d;\;
\xi(e)=1\}$ and the set of \textit{closed edges} by $\cC = \EE^d \setminus
\cO$. A site $v \in \ZZ^d$ \textit{percolates} if it belongs to an infinite
connected component of $\cO$. We say that $\cO$ \textit{percolates} if there
is a site $v$ that percolates.

A standard coupling shows that the model is monotone in $\lambda$, see
Section~\ref{sec:properties}. 
Thus, it is reasonable to expect that the model exhibits a non-trivial
phase transition in parameter $\lambda$.
Define
\begin{equation*}
\theta(p, \lambda, d)
    := \PP_{p, \lambda, d}(o\ \text{percolates})
    \quad \text{and} \quad
\lambda_c(p, d)
    := \sup \{\lambda \geq 0;\; \theta(p, \lambda, d)= 0\},
\end{equation*}
where $o$ denotes the origin in $\mathbb{Z}^d$.
We denote by $\pcs(d)$ and $\pcb(d)$ the critical points for site and bond
independent percolation on $\LL^{d}$.
It follows from \cite[Theorem~2.4]{beaton2019alignment} that
$\lambda_c(p, d) \geq \frac{p}{2d-1}> 0$ and
also that there is a universal $p^{\ast} > 0$ such that
\begin{equation*}
    \lambda_c(p, d) < 1
    \quad \text{if} \quad
    p < p^{\ast} \quad \text{or} \quad p > \pcs(d).
\end{equation*}
Since $\pcs(d) \sim \frac{1}{2d}$ and is strictly decreasing
on $d$ we know there is some $d_0(p^{\ast})$ such that for
$d \geq d_0(p^{\ast})$ we have $\lambda_c(p, d) \in (0, 1)$.
However, since we do not have a good control on $p^{\ast}$ it may
be the case that $p^{\ast} < \pcs(d)$.
Therefore, their result does not rule out that $\lambda_c(p, d)= 1$ on
$[p^{\ast}, \pcs(d)]$. In Theorem~\ref{teo:lambda_critical_non_trivial}
we bridge this gap providing non-trivial upper bounds on $\lambda_c(p,d)$
for every $p \in (0,1]$. For $d \geq 3$ it is not hard to find an upper
bound which holds uniformly over all the interval $(0,1]$ by a simple
comparison to Bernoulli percolation on the hexagonal lattice.
For the case $d=2$, that seems to be more delicate,
we employ renormalization methods in order to obtain the result.
\begin{teo}
\label{teo:lambda_critical_non_trivial}
Let $p_c^{\normalfont \hexagon} = 1 - 2 \sin(\pi/18)$ be the critical point for independent
bond percolation on the hexagonal lattice%
\footnote{For the hexagonal lattice the critical point is known exactly,
    see \cite[Chapter~3]{grimmett1999percolation}}.
For any $p \in (0,1]$ and $d \geq 3$ we have that
\begin{equation}
\label{eq:lambda_compared_hexagon_lattice}
\lambda_c(p, d) \le p_c^{\normalfont \hexagon}.
\end{equation}
Moreover, for every $p \in (0, 1]$ there is $\lambda_0(p) \in (0, 1)$
such that $\lambda_c(p, 2) \le \lambda_0$.
\end{teo}

\begin{figure}[ht]
\centering
\begin{tikzpicture}[scale=5,
    bgh/.style={blue, line width=1pt},
    new/.style={red, line width=2pt},
    aux/.style={dashed, gray}
    ]
    \draw[->] (-.15,   0) -- (1.1,   0) node[below] {$p$};
    \draw[->] (  0, -.15) -- (  0, 1.1) node[left] {$\lambda$};

    \coordinate (past0) at (.2,0) {};
    \coordinate (past1) at (.2,1) {};
    \coordinate (pcs0)  at (.4,0) {};
    \coordinate (pcs1)  at (.4,1) {};
    \coordinate (o)     at ( 0,0) {};
    \coordinate (pchex0) at ( 0,0.6527) {};
    \coordinate (pchex1) at ( 1,0.6527) {};
    \coordinate (pcb0)  at ( 0,0.4) {};
    \coordinate (pcb1)  at ( 1,0.4) {};
    \draw[aux] (1,1) -- (1,0);
    \draw[aux] (1,1) -- (0,1);
    \draw[aux] (pcb0) -- (pcb1);
    \draw[aux] (pcs0) -- (pcs1);
    \draw[aux] (past0) -- (past1);
    \draw[new] (pchex0) -- (pchex1);

    \draw (past1) ++(0,-1) -- ++(0,-.05) node[below] {$p^{\ast}$};
    \draw (pcs1)   ++(0,-1) -- ++(0,-.05) node[below] {$\pcs(d)$};
    \draw (1,0) -- (1,-.05) node[below] {$1$};

    \draw (0,1)   -- ++ (-.05,0) node[left] {$1$};
    \draw (pchex0) -- ++ (-.05,0) node[left] {$p_c^{\hexagon}$};
    \draw (pcb0)  -- ++ (-.05,0) node[left] {$\pcb(d)$};

    \draw[bgh] (o) -- (1, .3333); 
    \draw[bgh] plot [smooth]
    coordinates {(o) (0.05,0.15) (0.15,0.6527)}
    --(0.15,0.6527) -- (past1)
    --plot [smooth] coordinates {(pcs1) (0.5, 0.8) (0.6,0.6527)}
    -- plot [smooth] coordinates {(0.6,0.6527) (0.8,0.5) (pcb1)};

    \fill[black!50, opacity = .5] plot [smooth]
    coordinates {(o) (0.05,0.15) (0.15,0.6527)}
    --(0.6,0.6527)
    -- plot [smooth] coordinates {(0.6,0.6527) (0.8,0.5) (pcb1)}
    --(1, .3333)
    --cycle;
\end{tikzpicture}
\caption{For $d \geq 3$, Theorem~\ref{teo:lambda_critical_non_trivial}
    complements Figure~2.1 of \cite{beaton2019alignment} with an explicit upper
    bound; the critical curve is inside the highlighted region.
    For $d = 2$ we do not have explicit bounds but we ensure the
    critical curve is non-trivial in the interval $[p^*, p_c^{\text{site}}(2)]$.}
\label{fig:phase_transition_diagram}
\end{figure}
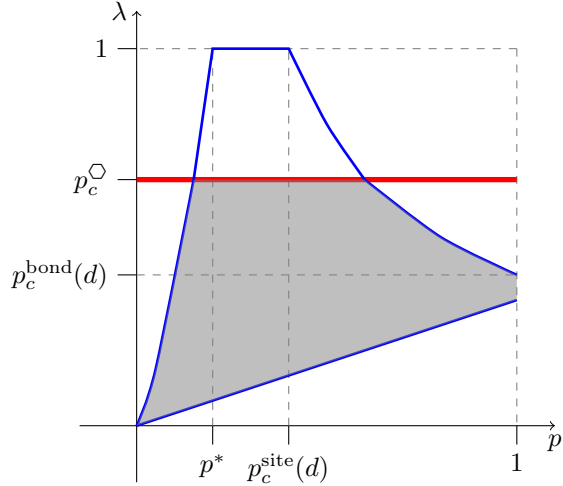

In~\cite{beaton2019alignment} the authors conjecture that the function
$p \mapsto \lambda_c(p, d)$ is non-increasing and continuous, as suggested by
numerical simulation. As already mentioned there, one difficulty in obtaining
results in this direction stems from the fact that it seems hard to couple the
model with different values of $p$ and still keep control of how the
connectivity changes, see Section~\ref{sec:properties}. Currently, there is
not much progress in proving strict monotonicity and continuity of
$\lambda_c(\cdot,d)$, however, our renormalization arguments can
be used to prove continuity as $p$ approaches $1$.

\begin{teo}
\label{teo:continuity_at_1}
Fix $d \geq 2$. 
For every $\lambda < \pcb(d)$ there exists $p_0(\lambda, d) \in (0, 1)$ such that
\begin{equation}
\label{eq:continuity_at_1}
\PP_{p, \lambda, d}(\text{$\cO$ percolates}) = 0
\quad \text{for every $p > p_0$.}
\end{equation}
In particular, the function $p \mapsto \lambda_c(p, d)$ is continuous at $p=1$.
\end{teo}
Theorem~\ref{teo:continuity_at_1} answers affirmatively a question posed in
\cite{beaton2019alignment} (see the paragraph preceding Remark~2.6 therein).

Let us say a word about the proofs of
Theorems~\ref{teo:lambda_critical_non_trivial} and \ref{teo:continuity_at_1}.
As mentioned above, the proof of Theorem~\ref{teo:lambda_critical_non_trivial}
for $d \geq 3$ is obtained by embedding the hexagonal lattice in $\ZZ^d$ in an
appropriate way. This comparison does not work for $d=2$ and to prove that
$\lambda_c(\cdot, 2)<1$ we develop a multiscale renormalization.
 
Renormalization is a classical technique that has been successfully applied to a
vast range of areas in mathematics and physics.
Instead of providing a complete account on the applicability of the method,
we would just like to mention that it has been successfully used to study
models in which correlations decay suitably with distance.
To mention a few examples, it has been employed to show the existence of
phase transition on dependent percolation e.g.\ random
interlacements~\cite{sznitman2010vacant} and Poisson cylinder
model~\cite{tykesson2012percolation}, to study the long-range behavior
of random walks on random environments~\cite{blondel18} and on detection
problems~\cite{baldasso2018can}.

The structure of the argument, as we use here, can be summarized in the
following steps.
\begin{itemize}
\item[\textbf{S1.}] (Recurrence inequality)
Define a sequence of scales $L_k$ increasing sufficiently fast
and consider a family of events $\{A_k(x);\; k \in \NN,\, x \in \ZZ^d\}$
where each $A_k(x)$ is measurable with respect to the restriction of the
process to a finite region having diameter of order $L_k$ around $x$.
Moreover the family is required to be \textit{cascading},
meaning that the occurrence of $A_{k+1}(x)$ implies the occurrence
of events $A_k(x')$ and $A_k(x'')$ for points $x'$ and $x''$
well-separated in scale $k$ but close enough in scale $k+1$.
Defining $q_{L_{k}}(p, \lambda, d) := \PP_{p, \lambda, d}(A_k(x))$
which does not depend on $x$ (due to translation invariance)
one can deduce that
\begin{equation}
\label{eq:general_recurrence_ineq}
q_{L_{k+1}}(p, \lambda, d)
    \le g\Bigl(\frac{L_{k+1}}{L_{k}}\Bigr) \cdot
        \bigl[
        q_{L_{k}}(p, \lambda, d)^2 + \text{error}_k(p, \lambda, d)
        \bigr],
\end{equation}
where $g$ is some positive real function and the error term is small due to the
decay of correlations and the choice of $L_k$.
In summary, inequality~\eqref{eq:general_recurrence_ineq} relates the
probabilities of the events $A_k$ at two successive scales by using
the structure of cascading events and
decay of correlations.

\item[\textbf{S2.}] (Inductive step)
Assuming that $q_{L_{k_{0}}}(p, \lambda, d) \le f(L_{k_{0}})$ for some choice
of $k_0(p, \lambda, d)$ and positive real function $f(x)$ tending to zero as
$x \to \infty$, we show inductively that
\begin{equation}
\label{eq:general_inductive_step}
q_{L_{k}}(p, \lambda, d) \le f(L_{k})
\quad \text{for every $k \geq k_0$}.
\end{equation}

\item[\textbf{S3.}] (Trigger)
Find a choice of $p, \lambda, d,  k_0$ so that
$q_{L_{k_{0}}}(p, \lambda, d) \le f(L_{k_{0}})$.
Once we fix such parameters, we can trigger the induction
in~\eqref{eq:general_inductive_step} and conclude that
$q_{L_{k}}(p, \lambda, d)$ indeed decays at least as fast as
$f(L_{k})$ as we increase $k$.
\end{itemize}

The proofs of Theorem~\ref{teo:lambda_critical_non_trivial} for $d=2$ and of
Theorem~\ref{teo:continuity_at_1}  follow essentially the structure above.  The
sequence of events $A_k(x)$ are either the existence of an open path or a dual
path crossing a box of size $L_k$ around $x$. Apart from the difference in the
definition of the events $A_k(x)$, the trigger step S3 is also performed
differently. For Theorem~\ref{teo:lambda_critical_non_trivial} we fix
$p \in (0,1)$ and take $\lambda$ sufficiently close to $1$, whereas for
Theorem~\ref{teo:continuity_at_1} we fix $\lambda < \pcb(d)$ and take $p$
sufficiently close to $1$. Our application is reasonably simple and showcases
the power of multiscale renormalization.  
This technique can be used in similar
percolation problems as long as it presents a good decay of correlations which
entails a small error term in \eqref{eq:general_recurrence_ineq}, see
Lemma~\ref{lema:decay_correlations}.

Throughout the text, we will use $c$ or $C$ to indicate positive constants
whose values may change from line to line.  Numbered constants like $c_0, c_1,
\ldots$ or $k_0, k_1,\ldots$ will have their values fixed at the first time
they appear and will remain fixed. We may indicate the dependence of constants
on other parameters, for instance, $c_i(d)$ is a constant whose value depend on
$d$ but not on any other parameter of the model.

Let us conclude this section by summarizing the structure of the paper.
On Section~\ref{sec:properties} we discuss some properties of the model:
monotonicity, uniqueness of the infinite cluster, the lattice condition, and
decay of correlations. 
On Section~\ref{sec:multi_scale_renormalization} we
give the multiscale framework that will be used for proving
Theorems~\ref{teo:lambda_critical_non_trivial} and~\ref{teo:continuity_at_1},
and then complete their proofs on
Sections~\ref{sub:phase_transition_for_alignment_percolation}
and~\ref{sub:continuity_at_p_1}, respectively.

\section{Properties of the independent model}
\label{sec:properties}

In this section we collect some useful properties fulfilled by
independent alignment. Some of them were were already mentioned
in~\cite{beaton2019alignment}.

Let us first describe some monotonicity properties. Since we are
concerned with percolation of the subgraph induced by the set of open edges,
we will compare two configurations $\xi,\, \xi' \in \Xi$ using
the natural partial order in $\Sigma$, that is,
\begin{equation*}
\xi \preccurlyeq \xi'
\quad \text{if and only if} \quad
\xi(e) \le \xi'(e),\ \forall e \in \EE^{d}.
\end{equation*}
Recall that $\PP_{p, \lambda}$ is the underlying probability distribution of
the model and denote by $\xi_{p,\lambda}$ a random element distributed as
$\mathbb{P}_{p,\lambda}$.
We may possibly omit one of the parameters in $\xi_{p,\lambda}$
when there is no risk of confusion.

For a fixed $p \in (0, 1]$ the $\xi_{p, \lambda}$ can be constructed
in a monotone way in $\lambda$ (with respect to this partial order)
using the standard coupling.
For a fixed $\lambda \in [0, 1]$ the picture is more complex. 
This was already mentioned on~\cite{beaton2019alignment}
but we elaborate on this discussion here.
\begin{prop}
\label{prop:no_monotonicity_p}
Fix $\lambda \in (0, 1]$ and denote by $\xi_p$ a random element
of $\Xi$ with distribution given by $\PP_{p, \lambda}$.
For any $0 < p_1 < p_2 \le 1$ we have that almost surely
neither $\xi_{\smash{p_1}}\preccurlyeq \xi_{\smash{p_2}}$ nor
$\xi_{\smash{p_2}} \preccurlyeq \xi_{\smash{p_1}}$.
\end{prop}

\begin{proof}
Recall that an event $A \subset \Sigma$
is said \textit{increasing} if $x \in A$ and $y \succcurlyeq x$
implies $y \in A$.
By Strassen's Theorem \cite[Theorem~2.4]{liggett2012interacting},
there exists a coupling of $(\xi_{p_1}, \xi_{p_2})$ such that
$\xi_{p_1} \succcurlyeq \xi_{p_2}$ if and only if
we have $\PP_{p_1}(A) \geq \PP_{p_2}(A)$ for any increasing event $A$.

The inequality $\mathbb{P}_{p_1} (A) \geq \mathbb{P}_{p_2}(A)$
actually holds for some increasing events.
For example, if for any finite $E \subset \EE^{d}$ we denote by $\xi(E)$
the restriction of $\xi$ to the edges of $E$ and write $\xi(E) \equiv 1$
if $\xi(e)=1$ for every edge in $E$ then we claim that
\begin{equation*}
\PP_{\smash{p_1}, \lambda} (\xi(E) \equiv 1)
    \geq \PP_{\smash{p_2}, \lambda} (\xi(E) \equiv 1)
    \quad \text{for every finite $E \subset \EE^{d}$}.
\end{equation*}
In order to prove this, we use a standard coupling for the underlying
site percolation configuration.
Let $\{U_{v};\; v \in \ZZ^{d}\}$ be a family of i.i.d.\ random variables with
the uniform distribution on $[0,1]$ and for $p \in (0, 1]$ define
$\eta_p \in \Omega$ by $\eta_p(v) := \I\{U_v \le p\}$ for every $v \in \ZZ^d$.
On $\Omega$ we have a partial ordering defined by $\eta \preccurlyeq \eta'$ if
and only if $\eta(v) \le \eta'(v)$ for every $v \in \ZZ^{d}$.
Notice that the coupling above makes $\eta_{\smash{p_1}} \preccurlyeq \eta_{\smash{p_2}}$.
Denote by $t(\eta_p, E)$ the number of feasible pairs of $\eta_p$
that contain at least one edge that belongs to $E$.
We have
\begin{equation*}
\PP_{p, \lambda}( \xi(E) \equiv 1)
    = \EE\bigl[ \PP_{p, \lambda}( \xi(E) \equiv 1 \mid \eta_p)\bigr]
    = \EE \bigl[ \lambda^{t(\eta_p, E)} \bigr].
\end{equation*}
Adding new sites to a configuration can only increase $t(\cdot, E)$. Thus, we have
\begin{equation*}
\PP_{\smash{p_1}, \lambda}( \xi(E) \equiv 1)
    = \EE \bigl[ \lambda^{t(\eta_{\smash{p_1}}, E)} \bigr]
    \geq \EE \bigl[ \lambda^{t(\eta_{\smash{p_2}}, E)} \bigr]
    = \PP_{\smash{p_2}, \lambda}( \xi(E) \equiv 1).
\end{equation*}
On the other hand, the inequality may fail for some increasing events $A$.
In fact, as already noted in~\cite{beaton2019alignment} the
increasing event 
\begin{equation*}
A \!= \!\{\text{$o$ has an open incident edge}\}
\end{equation*} 
has probability $p[1 - (1 - \lambda)^{2d}] + (1-p) [1 - (1 - \lambda)^{d}]$ which is
increasing in $p$.
\end{proof}

Conjecture~2.5 of \cite{beaton2019alignment} states that the function
$\lambda_c(\cdot, d)$ should be continuous and strictly increasing.
Proposition~\ref{prop:no_monotonicity_p} does not invalidate this
conjecture but shows that to compare the model for different values of
$p$ may not be a simple task.

Another known property of the model is that whenever $\cO$ percolates,
\ the infinite cluster is unique $\PP_{p, \lambda, d}$-a.s.
The proof of this result follows from the classical argument of
Burton and Keane~\cite{burton89},
and is part of~\cite[Theorem~2.4]{beaton2019alignment}.
This property may be useful, for instance to prove continuity of
the critical curve but we will not use it in this paper.

We now discuss other properties that were not presented
in~\cite{beaton2019alignment}.
Denote by $\mu$ the marginal of $\PP_{p, \lambda}$ on $\Sigma$.
Moreover, for $x, y \in \Sigma$ define $x \vee y, x \wedge y \in \Sigma$ as
\begin{equation*}
    (x \vee y)(e) := \max\{x(e), y(e)\}
    \quad \text{and} \quad
    (x \wedge y)(e) := \min\{x(e), y(e)\}
    \quad \text{for every $e \in \EE^{d}$}.
\end{equation*}
It is worth noticing that $\mu$ fails to satisfy the so-called 
\textit{lattice condition}
\begin{equation}
\label{eq:lattice_condition}
    \mu(x \vee y) \mu(x \wedge y) \geq \mu(x) \mu(y)
    \quad \text{for every $x, y \in \Sigma$.}
\end{equation}
Indeed, consider $E = \{\pm e_1, \pm e_2\} \subset \EE^{d}$ and
the configurations on $\{0, 1\}^{E}$ given by
\begin{equation*}
    x(e) = \I\{\pm e_1\}(e)
    \quad \text{and} \quad
    y(e) = \I\{e_1, e_2\}(e).
\end{equation*}
It is straightforward to check that
\begin{align*}
&\mu(x \vee y) \mu(x \wedge y) - \mu(x) \mu(y) \\
    &\hspace{5mm} = \bigl[p \lambda^3(1-\lambda)\bigr] \cdot \bigl[p \lambda
    (1-\lambda)^3\bigr]
    - \bigl[p \lambda^2(1 - \lambda)^2 + (1-p)\lambda (1 - \lambda)\bigr] \cdot
    \bigl[p \lambda^2 (1 - \lambda)^2\bigr] < 0.
\end{align*}
Proving~\eqref{eq:lattice_condition} is the usual strategy to show that
the measure $\mu$ satisfies the FKG inequality, also known as positive
association.
However, the FKG inequality is not equivalent to the lattice condition;
for background on positive association and its relation to the lattice
condition we refer to~\cite[Chapter~2]{grimmett2006random}.
This raises the question whether the model satisfies the FKG inequality.

Although independent alignment percolation model does not exhibit great
monotonicity properties, it does present fast correlation decay with distance. 
This fact will allow us to implement the multiscale renormalization approach
outlined in Section~\ref{sec:definition_of_the_model}.

For $x \in \ZZ^{d}$ and $L > 0$ let us define 
\begin{equation*}
B(x, L)
    := \bigl\{z \in \ZZ^{d};\; \norm{z-x}_{\infty} \le L\bigr\}
\quad \text{and} \quad
\partial B(x, L)
    := \bigl\{z \in \ZZ^{d};\; \norm{z-x}_{\infty} = \lfloor L \rfloor\bigr\}.
\end{equation*}
For $B \subset \ZZ^{d}$, we say that an event $A$
\textit{is supported on edges of}
$B$ if $A$ belongs to the smallest $\sigma$-algebra that makes
$\bigl\{\xi(e);\; e = \{u, v\} \in \EE^{d},\, u,v \in B\bigr\}$ measurable.

\begin{lema}[Decay of correlations]
\label{lema:decay_correlations}
Let $L > 0$ and $p \in (0,1)$.
For $x_1, x_2 \in \ZZ^{d}$ with $D = \norm{x_1 - x_2}_{\infty} > 2L$,
let $A_i$,  $i = 1,2$, be events supported on edges of $B_i := B(x_i, L)$,
respectively. Then,
\begin{equation}
\label{eq:decay_correlations}
\bigl|\Cov_{p, \lambda}(A_1, A_2)\bigr|
    \le 4 \cdot (2L+1)^{d-1} \cdot
    e^{-\alpha(p) \cdot (D- 2L )}
\end{equation}
with $\alpha(p) := - \log (1-p) > 0$.
\end{lema}

\begin{proof}
The assumption on $D$ implies $\dist(B_1, B_2) = D - 2 \lfloor L \rfloor \geq 1$
and, in particular, $B_1 \cap B_2 = \varnothing$. Choose the
smallest $1 \le t \le d$ such that $D = |x_{1,t} - x_{2, t}|$ and
notice that we can separate $B_1$ from $B_2$ by a hyperplane orthogonal to $e_t$
the $t$-th vector of the canonical basis of $\RR^{d}$. Denote by
$\pi_t(B)$ the projection of a set $B \subset \RR^{d}$ into the subspace
$e_t^{\perp}$ (perpendicular to $e_t$) and define
\begin{equation*}
\Pi_t = \pi_t(B_1) \cap \pi_t(B_2).
\end{equation*}
If we have $\Pi_{t} = \varnothing$ then $A_1$ and $A_2$ are independent and
their covariance is zero. Otherwise, for every $z \in \Pi_t$ define
$I_z = [a_z, b_z] \cap \ZZ^{d}$ as the unique line
segment supported on the line $\pi_t^{-1}(z)$ with $a_z \in B_1$,
$b_z \in B_2$ and such that $(I_z\setminus\{a_z,b_z\}) \cap (B_1\cup B_2) = \varnothing$.
Define the event
\begin{equation}
\label{eq:decoupling_event}
C = \{ \exists\, z \in \Pi_t \text{ such that }\omega(v) = 0,\, \forall v \in I_z\}
\end{equation}
and notice that on $C^{c}$ the state of edges on $B_1$ and $B_2$
are independent. Indeed, for $\omega \in C^{c}$ notice that if $F_i(\omega)$
denotes the feasible edges intersecting $B_i$ and
$f_i = \{u_i, v_i\} \in F_i(\omega)$ then
$|f_1 \cap f_2| \le 1$. Thus, omitting the dependency on $p$ and $\lambda$
we have that
\begin{equation*}
\PP(A_1 \cap A_2 \cap C^{c})
    = \PP(A_1 \cap A_2 \mid C^{c}) \PP(C^{c})
    = \PP(A_1 \mid C^{c}) \PP(A_2 \mid C^{c}) \PP(C^{c}).
\end{equation*}
Using twice that for any event $E$ one has
$\bigl|\PP(E) - \PP(E \cap C^{c})\bigr| \le \PP(C)$,
we can bound
\begin{align*}
\bigl| \PP(A_1) \PP(A_2) &- \PP(A_1 \cap C^{c}) \PP(A_2 \cap C^{c}) \bigr| \\
    &= \bigl|
    \PP(A_1) [\PP(A_2) - \PP(A_2 \cap C^{c})] + 
    \PP(A_2 \cap C^{c}) [\PP(A_1) - \PP(A_1 \cap C^{c})]
    \bigr| \\
    &\le (1 + \PP(C^{c})) \cdot \PP(C),
\end{align*}
which leads to the estimate
\begin{align*}
\bigl|\Cov(A_1, A_2)\bigr|
    &= \bigl| \PP(A_1 \cap A_2) - \PP(A_1)\PP(A_2) \bigr| \\
    &\le \PP(C) + (1 + \PP(C^{c})) \PP(C) + 
        \bigl|
        \PP(A_1 \cap A_2 \cap C^{c}) - \PP(A_1 \cap C^{c}) \PP(A_2 \cap C^{c})
        \bigr| \\
    &= (2 + \PP(C^{c})) \PP(C) + 
        \bigl|
        \PP(A_1 \mid C^{c}) \PP(A_2 \mid C^{c}) \cdot
        [\PP(C^c) - \PP(C^{c})^{2}]
        \bigr| \\
    &\le 2 \PP(C) + \PP(C^{c}) \PP(C) + 
        \PP(C^c) - \PP(C^{c})^{2} \\
    &= 4 \PP(C) - 2 \PP(C)^2.
\end{align*}
Finally, we use the union bound and the fact that $|\Pi_t| \le (2L+1)^{d-1}$ to
estimate
\begin{equation*}
    \PP(C) \le (2L+1)^{d-1} \cdot (1-p)^{D - 2L}. \qedhere
\end{equation*}
\end{proof}

\begin{remark}
Lemma~\ref{lema:decay_correlations} is stated for $\PP_{p, \lambda}$. However,
it can be adapted for the `one-choice model' and more generally
for any measures $P^{\omega}$ on the feasible pairs $F(\eta)$
satisfying that random variables
$\{\xi(f_i);\; i \in I\}$ are mutually independent for
any endpoint-disjoint family $\{f_i;\; i \in I\} \subset F(\eta)$.
To see why this is true, just consider instead of the event $C$ appearing
in \eqref{eq:decoupling_event} the event
\begin{equation*}
\tilde{C}
    = \{\exists\, f_1 \in F_1(\omega), f_2 \in F_2(\omega) \text{ such that }
        |f_1 \cap f_2| \geq 1\}.
\end{equation*}
To prevent $\tilde{C}$ from occurring we just have to ensure each line $I_z$ with
$z \in \pi_t(B_1) \cup \pi_t(B_2) =: \tilde{\Pi}_t$ contains
$2$ sites in $\eta(\omega)$ lying in the region between the two boxes.
The decay of correlations will follow once one notes that
$|\tilde{\Pi}_t| \le 2(2L + 1)^{d-1}$ and
changing $\alpha(p)$ accordingly. Moreover, we notice that the bound on
\eqref{eq:decay_correlations} is independent of $\lambda$.
\end{remark}

\section{Multiscale renormalization}
\label{sec:multi_scale_renormalization}

In this section we build a multiscale renormalization scheme that allow us
to prove Theorem~\ref{teo:lambda_critical_non_trivial} in the case $d=2$
and Theorem~\ref{teo:continuity_at_1}.
A key step will be the use of the correlation decay in
Lemma~\ref{lema:decay_correlations}.

We start by defining the sequence of scales $L_0, L_1, \ldots$
along which we analise the model. Let
\begin{equation}
\label{eq:sequence_scales}
\text{$L_0 := 10^4$ \,\,\,\,\, and \,\,\,\,\, $L_{k+1} = L_k^{3/2}$,}
\end{equation}
that is, $L_k = L_0^{(3/2)^k}$ is a sequence growing super-exponentially fast.
This same sequence is used in the proofs of both
Theorem~\ref{teo:lambda_critical_non_trivial} in the case $d=2$ and
Theorem~\ref{teo:continuity_at_1}.
The family of cascading events
$\{A_k(x)\}_{x \in \mathbb{Z}^d, k\in \mathbb{Z}_+}$
needs to be defined properly for each proof, but $A_k(x)$
will always be supported in the annular region
$B(x, 10L_{k}) \setminus B(x, L_{k})$. 
However once these events are defined, Steps {\bf S1} and {\bf S2}
in Section~\ref{sec:definition_of_the_model} can be carried on
exactly the same way in both proofs.

For each $k$, let $\cL_k^{1} \subset \partial B(o, L_{k+1})$
and $\cL_k^{2}\subset \partial B(o, 5 L_{k+1})$ be minimal collections
of points such that
\begin{equation*}
\partial B(o, L_{k+1}) \subset \bigcup_{x \in \cL_n^{1}} B(x, L_k)
\quad \text{and} \quad
\partial B(o, 5 L_{k+1}) \subset \bigcup_{x \in \cL_n^{2}} B(x, L_k).
\end{equation*}
The following lemma gives bounds on the size of $\cL_{k}^{i}$.
\begin{lema}
\label{lema:coverings}
There are constants $c = c(d)$ and $C = C(d)$ such that for $i = 1, 2$ we have
\begin{equation}
\label{eq:coverings}
c \cdot \Bigl(\frac{L_{k+1}}{L_{k}}\Bigr)^{d-1}
    \le |\cL^{i}_{k}|
    \le C \cdot \Bigl(\frac{L_{k+1}}{L_{k}}\Bigr)^{d-1}.
\end{equation}
\end{lema}

\begin{proof}
We prove the result for $\cL^{1}_{k}$ and the same reasoning applies for
$\cL_{k}^{2}$. Notice that any $B(x, L_k)$ with
$x \in \partial B(o, L_{k+1})$ satisfies
\begin{equation}
\label{eq:coverings_one_ball}
\lfloor L_{k} \rfloor^{d-1}
    \le |B(x, L_{k}) \cap \partial B(o, L_{k+1})|
    \le d \cdot (2 \lfloor L_{k} \rfloor + 1)^{d-1}.
\end{equation}
Since $\{B(x, L_{k});\; x \in \cL_{k}^{1}\}$ covers
$\partial B(o, L_{k+1})$, the upper bound on~\eqref{eq:coverings_one_ball}
implies that
\begin{equation*}
2d \cdot (2 \lfloor L_{k+1} \rfloor + 1)^{d-1}
    = |\partial B(o, L_{k+1})|
    \le |\cL_{k}^{1}| \cdot d \cdot (2 \lfloor L_{k} \rfloor + 1)^{d-1}.
\end{equation*}
and thus the lower bound on~\eqref{eq:coverings} follows. For the upper bound
consider $\cN_{0} = \varnothing$ and, inductively, while
$\{B(x, \frac{1}{2} L_k);\; x \in \cN_{j}\}$ is a disjoint collection and there
is some
\begin{equation*}
y \in \partial B(o, L_{k+1})
    \setminus \bigcup_{x \in \smash{\cN_{j}}} B(x, \tfrac{1}{2} L_k),
\end{equation*}
define $\cN_{j+1} := \cN_{j} \cup \{y\}$. This process ends in finitely many
steps, producing a set $\cN$ satisfying
\begin{equation*}
2d \cdot (2 \lfloor L_{k+1} \rfloor + 1)^{d-1}
    = |\partial B(o, L_{k+1})|
    \geq |\cN| \cdot \lfloor \tfrac{1}{2} L_{k} \rfloor^{d-1},
\end{equation*}
using~\eqref{eq:coverings_one_ball}.
Finally, just notice that we can take $\cL_{k}^{1} \subset \cN$,
since for any $y \in \partial B(o, L_{k+1})$ we have
\begin{equation*}
    \dist(y, B(x, \tfrac{1}{2} L_{k})) < \tfrac{1}{2} L_k
    \quad \text{for some $x \in \cN$} \quad \implies \quad 
    \norm{x - y}_{\infty} \le L_{k}. \qedhere
\end{equation*}
\end{proof}
\noindent
We require two properties for our events $A_k(x)$.
\begin{itemize}
\item[\textbf{P1.}]
    $\PP_{p, \lambda, d} (A_k(x))$ does not dependent upon
    the choice of $x$ which allows us to define
    \begin{equation*}
        q_k(p, \lambda, d) := \PP_{p, \lambda, d} (A_k(o)).
    \end{equation*}

\item[\textbf{P2.}] The occurrence of event $A_{k+1}(o)$ implies
    that there are $x_i \in \cL_{k}^{i}$ such that $A_{k}(x_i)$
    also occur, for $i=1,2$. Heuristically, we say that events
    $A_{k}(x)$ are cascading since if an event occurs on scale
    $k+1$ it must have occurred on two well-separated regions for
    the previous scale.
\end{itemize}
Properties {\bf P1} and {\bf P2} imply the validity of Steps
{\bf S1} and {\bf S2}.
\begin{lema}[Recurrence Inequality]
\label{lema:recurrence_inequality}
Let $\{A_k(x);\; x \in \ZZ^{d}, k \in \NN\}$ be a collection of events on
$\Xi$ satisfying Properties P1 and P2. 
Then, 
\begin{equation}
\label{eq:recurrence_inequality_both}
q_{k+1}(p, \lambda, d)
    \le c_{0}(d) \cdot L_{k}^{d-1}  q_k(p, \lambda, d)^2 +
        c_{1}(d) \cdot L_k^{2d-2}
        e^{- 3\alpha(p) \cdot \smash{L_{k}^{3/2}}},
\end{equation}
where $c_0$ and $c_1$ are positive constants depending only on $d$.
\end{lema}

\begin{proof}
The boxes $B(x_i, 10 L_{k})$ given by property {\bf P2} 
are well-separated since
\begin{equation*}
\norm{x_1 - x_2}_{\infty}
    \geq \lfloor 5 L_{k+1} \rfloor - \lfloor L_{k+1} \rfloor
    \geq 4 \lfloor L_{k+1} \rfloor
    \quad
    \text{for every $x_i \in \cL_k^{i}$}.
\end{equation*}
Using the fact that $A_k(x)$ is supported on edges inside $B(x, 10 L_k)$,
we have by Lemma~\ref{lema:decay_correlations} that
there exists a constant $c_2 = c_2(d)>0$ such that
\begin{equation*}
|\Cov_{p, \lambda} (A_{k}(x_1), A_{k}(x_2))|
    \le c_2 \cdot L_k^{d-1} \cdot
        e^{- \alpha(p) (4 \lfloor L_{k+1} \rfloor - 20 L_k)}
    \le c_2 \cdot L_k^{d-1} \cdot e^{- 3 \alpha(p) L_{k+1}}
\end{equation*}
for every $k$.
Lemma~\ref{lema:coverings} implies that
$|\cL_k^{i}|
\le c_3 \cdot \frac{L_{k+1}}{L_k}
= c_3 \cdot L_k^{(d-1)/2}$ for some constant $c_3(d) >0$. Thus, we can write
\begin{align}
q_{k+1}(p, \lambda, d)
    &\le \sum_{\smash{x_1} \in \smash{\cL_k^{1}},\, \smash{x_2} \in \smash{\cL_k^{2}}}
        \PP_{p, \lambda, d} (A_{k}(x_1) \cap A_{k}(x_2)) \nonumber \\
    &\le |\cL_{k}^{1}| \cdot |\cL_{k}^{2}| \cdot 
        \bigl[q_{k}(p, \lambda, d)^2 + c_{2}(d) \cdot L_{k}^{d-1}
        e^{- 3 \alpha(p) \smash{L_{k}^{3/2}}}\bigr] \nonumber \\
    &\le c_{3}(d)^2 \cdot L_{k}^{d-1}  q_k(p, \lambda, d)^2 +
       c_{3}(d)^2\cdot c_{2}(d) \cdot L_k^{2d-2} e^{- 3\alpha(p) \cdot \smash{L_{k}^{3/2}}}.
\end{align}
The proof is finished by defining $c_{0} = c_{3}^{2}$ and $c_{1}=c_{3}^2\cdot c_{2}$.
\end{proof}

\begin{lema}[Inductive step]
\label{lema:step_s2}
Let $\{A_k(x);\; x \in \ZZ^{d}, k \in \NN\}$ be a collection of events on
$\Xi$ satisfying Properties P1 and P2. 
There is $k_0 = k_0(p, d)$ such that if for some $k_1 \geq k_0$ one has
\begin{equation}
    \label{eq:inductive_step_k1}
    q_{k_1}(p, \lambda, d) \le L_{k_1}^{-2d}
\end{equation}
then
\begin{equation}
    \label{eq:inductive_step}
    q_{k}(p, \lambda, d) \le L_{k}^{-2d} \,\,\,\,\, \text{holds for every $k \geq k_1$}.
\end{equation}
\end{lema}

\begin{proof}
Choose $k_0(d, p)$ as the smallest integer $k$ such that
\begin{equation*}
c_{0}(d) L_k^{-1}
    \le \frac{1}{2}
\quad \text{and} \quad
c_{1}(d) \cdot L_{k}^{5d-2} e^{- 3\alpha(p) \cdot L_{k}^{3/2}}
    \le \frac{1}{2},
\end{equation*}
where $c_0(d), c_1(d)$ are given by Lemma~\ref{lema:recurrence_inequality}.
This is always possible since both left-hand sides tend to zero when
$k \to \infty$. 
Suppose there is some $k(p, \lambda, d) \geq k_0$ satisfying the
inequality in~\eqref{eq:inductive_step}. 
Then, we can write
\begin{align*}
\smash{\frac{q_{k+1}(p, \lambda, d)}{L_{k+1}^{-2d}}}
    &\le c_{0}(d) \cdot L_{k}^{d-1+3d}  q_k(p, \lambda, d)^2 +
        c_{1}(d) \cdot L_{k}^{2d - 2 + 3d}
        e^{- 3\alpha(p) \cdot L_{k}^{3/2}} \\
    &\le c_{0}(d) \cdot L_{k}^{4d-1} L_{k}^{-4d} +
        c_{1}(d) \cdot L_{k}^{5d - 2}
        e^{- 3\alpha(p) \cdot L_{k}^{3/2}} \\
    &\le 1,
\end{align*}
by our choice of $k_0$.
This means that the inequality in \eqref{eq:inductive_step}
carries on to the next scale $k+1$. The result follows by induction.
\end{proof}

\subsection{Phase transition for alignment percolation}
\label{sub:phase_transition_for_alignment_percolation}

We first consider the case $d \geq 3$, in which no renormalization is needed.
Indeed, a construction that appeared in~\cite{hilario2019bernoulli} fits
very nicely to this case.

\begin{proof}
[Proof of Theorem~\ref{teo:lambda_critical_non_trivial}, case $d \geq 3$]
We show that it is possible embed the hexagonal lattice in $\ZZ^d$ in such
a way that the state of their edges are independent. 
Notice that it suffices to prove the claim on the case $d=3$,
since for larger $d$ one can run the same argument on the subset
$\{x \in \ZZ^d;\; x_i = 0, \forall\  4 \le i \le d\}$.

Thus, let us consider $d =3$ and for a fixed $k \in \ZZ$ we write
\begin{equation*}
    V_1 := \big\{(x,y,z) \in \ZZ^3;\; x+y+z = k\big\}
    \quad \text{and} \quad
    V_2 := \big\{(x,y,z) \in \ZZ^3;\; x+y+z = k+1\big\}.
\end{equation*}
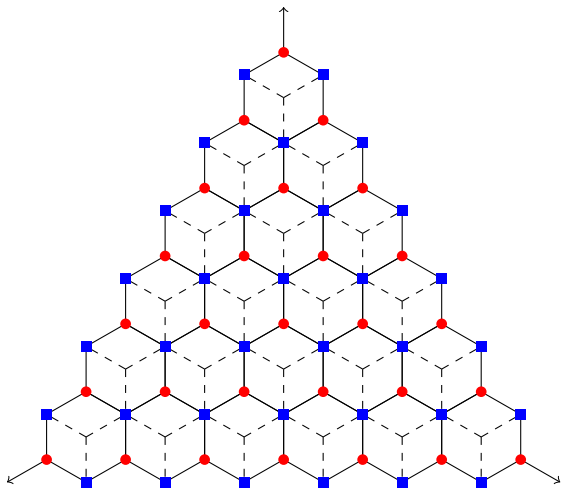
\begin{figure}[ht]
\centering
\begin{tikzpicture}[scale=.6,
    x={(210:1cm)}, y={(330:1cm)}, z={(0cm, 1cm)},
    plane6/.style={fill, red, minimum size=4pt, outer sep=0pt,
                    inner sep=0pt, circle},
    plane7/.style={fill, blue, minimum size=4pt, outer sep=0pt,
                    inner sep=0pt},
    ]
\draw[->] (6,0,0) -- (7,0,0);
\draw[->] (0,6,0) -- (0,7,0);
\draw[->] (0,0,6) -- (0,0,7);

\foreach \z in {0, ..., 5}{
    \pgfmathsetmacro\end{5-\z}
    \foreach \x in {0, ..., \end}{
        \draw (\x, 5-\x-\z, \z) {++ (0,1,0)  -- ++(0,0,1) -- ++(0,-1,0) 
         -- ++(1,0,0) -- ++(0,0,-1) -- ++(0,1,0) -- cycle};
         \draw[dashed] (\x, 5-\x-\z, \z) -- ++(1,1,0);
         \draw[dashed] (\x, 5-\x-\z, \z) -- ++(0,1,1);
         \draw[dashed] (\x, 5-\x-\z, \z) -- ++(1,0,1);
    };
};

\foreach \z in {0, ..., 6}{
    \pgfmathsetmacro\end{6-\z}
    \foreach \x in {0, ...,\end}{
        \node[plane6] at (\x, 6-\x-\z, \z) {};
    };
};

\foreach \z in {1, ..., 6}{
    \pgfmathsetmacro\end{7-\z}
    \foreach \x in {0, ...,\end}{
        \node[plane7]
        at (\x, 7-\x-\z, \z) {};
    };
};
\foreach \x in {1, ..., 6}{
    \node[plane7] at (\x, 7-\x, 0) {};
};
\end{tikzpicture}
\caption[Copy of the hexagonal lattice with the property that
        every edge is independent.]
    {
    Copy of the hexagonal lattice with
    the property that every edge is independent. In the picture above, we used
    $k=6$ to obtain
    $V_1 = \{%
    \begin{tikzpicture}
        \fill[red] (0,0) circle (2pt);
    \end{tikzpicture}
    \}$ and
    $V_2 = \{%
    \begin{tikzpicture}%
        \draw[fill, blue] (0,0) rectangle (1ex, 1ex);
    \end{tikzpicture}
    \}$. Solid lines represent edges of the subgraph and dashed lines are
    present to help visualizing the lattice as a subset of $\RR^3$.
    }
\label{fig:hexagonal_lattice_immersion}
\end{figure}

The desired subgraph of $\LL^3$ is obtained by considering the vertex set $V_1 \cup V_2$ and the edges $uv$ with $u \in V_1$, $v \in V_2$ and $\norm{u-v}_\infty = 1$.
As shown in Figure~\ref{fig:hexagonal_lattice_immersion}, this defines a subgraph of $\mathbb{Z}^3$ that is isomorphic to the hexagonal lattice.
Under the measure $P^{\omega}_{\lambda}$, the state of each edge is independent since a line parallel to one of the canonical directions intersects the planes $V_1$ and $V_2$ in precisely one point each.
Therefore, the alignment percolation restricted to such a graph is just independent Bernoulli percolation whose critical point equals $p^{\hexagon} = 1- 2\sin(\pi/8)$.
\end{proof}

We now focus on the case $d=2$.
Consider the dual graph $\dual{\LL^{2}}$ of $\LL^{2}$, defined as the graph with vertex set $\ZZ^{2} + (1/2, 1/2)$ and edge set $\dual{\EE^{2}}$ connecting sites at Euclidean distance~$1$.
Therefore, $\dual{\LL^{2}}$ and $\LL^{2}$ are isomorphic.
Moreover, to each edge $e^*$ in $\dual{\EE^{2}}$ it corresponds a unique edge $e$ in $\EE^{2}$, so that $e$ and $e^*$ intersect at right angle.
Given any configuration $\xi \in \Xi$, we denote $\cO^{\ast}$ (resp.\ $\cC^{\ast}$) the set of dual edges whose corresponding primal edge is open (resp.\ closed). 
Notice however that, unlike when we consider independent bond percolation, for alignment percolation the distributions of $\cO^{\ast}$ and $\cO$ are not the same. 
Notice also that any event can be defined in terms of the statuses of either the primal or the dual edges.

Define for $x \in \ZZ^2$ events
\begin{equation*}
A_k(x) := \big\{
    \text{there is an open circuit on $B(x, 10L_k) \setminus B(x, L_k)$
    surrounding $B(x, L_k)$}
\big\}^{c},
\end{equation*}
which can be seen as the event on which there exists a dual path of $\cC^{\ast}$ edges from the inside of $B(x, L_k)$ to the outside of $B(x, 10L_k)$.
The key fact is that the events $A_{k}(x)$ satisfy Properties \textbf{P1} and \textbf{P2}, allowing us to use Lemmas~\ref{lema:recurrence_inequality} and~\ref{lema:step_s2}.

\begin{lema}[Trigger for Theorem~\ref{teo:lambda_critical_non_trivial}]
\label{lema:step_s3_transition}
For any $p \in (0,1)$ there are $k_0(p) > 0$ and $\lambda_0(p) \in (0, 1)$
such that for every $\lambda \geq \lambda_0$ and $k \geq k_0$
\begin{equation*}
    q_k(p, \lambda) \le L_k^{-4}.
\end{equation*}
\end{lema}

\begin{proof}
Let $k_0(p)$ be given as in Lemma~\ref{lema:step_s2}.
Now, let us check that we can take $\lambda$ sufficiently close to
$1$ in order to ensure that $q_{k_0}(p, \lambda) \le L_{k_0}^{-4}$.
Let $N(k_0)$ be the total number of edges from $\EE^d$ with some extremity in $B(o, 10 L_{k_0})$ and define $\lambda_0 = (1-L_{k_0}^{-4})^{N(k_0)^{-1}}$.
Notice that we can write
\begin{equation*}
q_{k_0}(p, \lambda)
    = \PP_{p, \lambda}(A_{k_0}(o))
    \le 1 - \EE_{p, \lambda} \bigl[ \PP_{p, \lambda} (\text{all edges inside
    $B(o, 10 L_{k_0})$ are open} \mid \eta) \bigr].
\end{equation*}
Since $k_0$ is fixed, the number of edges of $F(\eta)$ that 
have some extremity inside $B(o, 10 L_{k_0})$ is bounded from above by $N(k_0)$. 
Thus, we have for every $\lambda\geq \lambda_0$,
\begin{equation*}
q_{k_0}(p, \lambda)
    \le 1 - \lambda^{N(k_0)} \leq 1-\lambda_0^{N(k_0)} = L_{k_0}^{-4}.
\end{equation*}
The result follows from Lemma~\ref{lema:step_s2}.
\end{proof}

\begin{proof}
[Proof of Theorem~\ref{teo:lambda_critical_non_trivial}, case $d=2$]
We now show that for any $\lambda \geq \lambda_0(p)$ we have that $\cO$  percolates $\PP_{p, \lambda}$-a.s. 
Since we are working on the plane, we have that if $\cO$ does not percolate there must be a sequence $\gamma_n$ of disjoint circuits in $\dual{\LL^{2}}$ that go
around the origin with $\gamma_n \subset \cC^{\ast}$ and
\begin{equation}
\label{eq:circuits_gamma_n}
    \dist\big(o, \gamma_n \cap (\RR^{+} \times \{0\})\big) \to \infty.
\end{equation}
For each $k \geq 0$ consider the points
$\{x_{k, i}\} \subset \RR^{+} \times \{0\}$ defined by
\begin{equation*}
    x_{k, i} := (\lceil 10 L_{k} \rceil + (i-1) 2 \lfloor L_{k} \rfloor,\ 0)
\quad \text{for $1 \le i \le s_k$},
\end{equation*}
where $s_k$ is determined below. Notice that
\begin{equation*}
    B(x_{k, i}, 10 L_k) \subset \RR^{+} \times \RR
    \quad \text{and} \quad
    \text{boxes $B(x_{k, i}, L_k)$ and $B(x_{k, i+1}, L_k)$ are adjacent.}
\end{equation*}
Choosing $s_k$ as the minimum index $i$ such that $x_{k, i} \geq x_{k+1, 1}$ ensures that the family of boxes $\{B(x_{k, i},L_k);\; k \geq 0, 1 \le i \le s_k\}$ covers the half-line
$(\lfloor 10 L_0 \rfloor, \infty) \times \{0\}$.
Moreover, the sequence $s_k$ satisfies
\begin{equation*}
s_k \le 5 \cdot \frac{L_{k+1}}{L_k} = 5 L_{k}^{1/2}.
\end{equation*}
Choose some ordering $\tilde{A}_l$ for the collection of events
$\{A_{k}(x_{k, i});\; k \geq 0, 1 \le i \le s_k\}$.
The existence of the sequence of circuits $\gamma_n \subset \cC^{\ast}$
satisfying \eqref{eq:circuits_gamma_n} implies
\begin{equation*}
\PP_{p, \lambda} \big(\{\cO\ \text{percolates}\}^{c}\big)
    \le \PP_{p, \lambda} (\tilde{A}_l,\ \text{i.o.}).
\end{equation*}
However, we have by Lemma~\ref{lema:step_s3_transition} that
\begin{equation*}
\sum_{l \geq 1} \PP_{p, \lambda} (\tilde{A}_l)
    = \sum_{k\geq 0} \sum_{i=1}^{s_k} q_k(p, \lambda)
    \le \sum_{k \geq 0} 5 \cdot L_{k}^{-7/2}
    < \infty
\end{equation*}
and Borel-Cantelli lemma implies that
$\PP_{p, \lambda} (\tilde{A}_l,\ \text{i.o.}) = 0$. We conclude that for
$\lambda \geq \lambda_0(p)$ the set $\cO$ percolates $\PP_{p, \lambda}$-a.s.
\end{proof}

\begin{figure}
\centering
\begin{tikzpicture}[scale=.13]
 \draw (-1, 0) -- (82,0) node[above,pos=.96] {$...$};
 \draw (0,1) -- (0,-1) node[below] {$o$};
 \draw (10.5,1) -- (10.5,-1) node[below] {$10 L_{0}$};
 \clip (-2,-20) rectangle (80,20);
 
 \begin{scope}[shift={(10,0)}]
 \foreach \x in {0, 1, 2}{
   \foreach \y in {0,...,10}{
   \begin{scope}[shift={({9*(pow(3,\x) - 1)/2},{-pow(3,\x)/2})}]
    \draw ( {\y * pow(3,\x)} , 0) rectangle ({ (\y + 1)* pow(3,\x)}, {pow(3,\x)});
   \end{scope}
   }
 } 
 \end{scope}

\fill[black!50, opacity = .5] (34, -1.5) rectangle (37, 1.5);
\draw[dashed] (22.5, -13) rectangle (48.5, 13);
 
 \draw[thick] plot [smooth cycle] coordinates {(-10,0) (0,-18) (10,-15) (17,0) (10,8) (0,15)};
 \node at (6,15) {$\gamma_1$};
 \draw[thick] plot [smooth cycle] coordinates {(-20,0) (0,-40) (20,-30) (35,0) (20,45) (0,30)};
 \node at (35,15) {$\gamma_2$};
  \draw[thick] plot [smooth cycle] coordinates {(-30,0) (0,-60) (30,-45) (54,0) (30,40) (0,45)};
 \node at (52,15) {$\gamma_3$};
\end{tikzpicture}
\caption{Multiscale renormalization for
    Theorem~\ref{teo:lambda_critical_non_trivial}, case $d=2$. If $\cO$ does not percolate,
    circuits $\gamma_n \subset \cC^{\ast}$ must intersect boxes $B(x_{k,i}, L_k)$
    with large $k$, as exemplified by the highlighted box.}
\label{fig:multiscale_renormalization}
\end{figure}
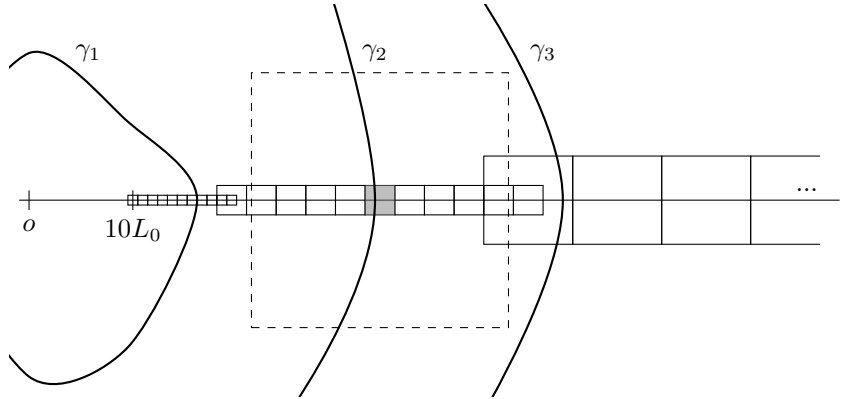

\subsection{\texorpdfstring{Continuity at $p=1$}{Continuity at p=1}}
\label{sub:continuity_at_p_1}

We now employ a similar renormalization argument to prove
Theorem~\ref{teo:continuity_at_1}.
The sequence of scales we choose is still~\eqref{eq:sequence_scales}, 
but now we consider events
\begin{align*}
A_k(x) 
    &= \big\{
    \text{$B(x, L_k)$ is connected to $\partial B(x, 10L_{k})$ by edges in $\cO$}
    \big\} \\
    &=: \big\{B(x, L_k) \leftrightarrow \partial B(x, 10L_{k})\big\}.
\end{align*}
Once again, it is straightforward to check that events $A_{k}(x)$ satisfy
properties \textbf{P1} and \textbf{P2}.

\begin{proof}[Proof of Theorem~\ref{teo:continuity_at_1}]
Fix $\lambda < \pcb(d)$. We first need to complete the trigger for
events $A_{k}(x)$. We can assume without loss of generality that
$p \in [1/2, 1]$ and notice that we can adapt the argument in
Lemma~\ref{lema:step_s2} to get $k_0$ depending only on $d$,
since the error term
in \eqref{eq:recurrence_inequality_both} is decreasing in $p$ as can be shown by 
using the fact that $\alpha(p) := - \log (1-p) > 0$, as given by Lemma \ref{lema:decay_correlations}.
Denoting by $V = V(d, k)$ the set of vertices of
$B(o, 10 L_{k})$ and by $\PPb_{\lambda}$ the independent
bond percolation measure on $\LL^{d}$, we can write
\begin{align*}
q_{k}(p, \lambda, d)
    &\le \PP_{p, \lambda}( \xi(V) \not\equiv 1) +
        \PP_{p, \lambda}
        \bigl(
            \xi(V) \equiv 1,
            B(o, L_{k}) \leftrightarrow \partial B(o, 10 L_{k})
        \bigr) \\
    &= (1 - p^{|V|}) + p^{|V|} \cdot
        \PPb_{\lambda}
        \bigl(
        B(o, L_{k}) \leftrightarrow \partial B(o, 10 L_{k})
        \bigr).
\end{align*}
Since $\lambda < \pcb(d)$ we have exponential decay for the
radius of an open cluster, cf.
\cite[Theorem~5.4]{grimmett1999percolation}.
Thus, there is a positive constant $\psi(\lambda)$ such that
\begin{equation*}
\PPb_{\lambda}
\bigl(
B(o, L_{k}) \leftrightarrow \partial B(o, 10 L_{k})
\bigr)
\le |\partial B(o, L_{k})| \cdot
        \PPb_{\lambda}\bigl(0 \leftrightarrow \partial B(o, 9L_{k})\bigr)
\le c(d)L_k^{d-1} \cdot \exp[- \psi(\lambda) L_{k}].
\end{equation*}
Let us now pick $\tilde{k}_0(\lambda, d) > 0$ such that
\begin{equation*}
c(d)L_k^{d-1} \cdot \exp[- \psi(\lambda) L_{k}]
    \le \tfrac{1}{2} L_{k}^{-2d}
    \quad \text{for $k \geq \tilde{k}_0(\lambda, d)$}.
\end{equation*}
Taking $k_1(\lambda, d) = \max\{\tilde{k}_0(\lambda, d), k_0(d)\}$, we have
\begin{equation}
\label{eq:trigger_s3_continuity_2}
q_{k_{1}}(p, \lambda, d)
    \le \big(1 - p^{|V(k_1)|}\big) + \tfrac{1}{2} L_{k_1}^{-2d}.
\end{equation}
Since $\lambda$ and $d$ are fixed we can pick $p_0=p_0(\lambda, d)$
such that $1 - p^{|V(k_1)|} \le \tfrac{1}{2} L_{k_1}^{-2d}$,
for every $p \geq p_0(\lambda, d)$.
Plugging into \eqref{eq:trigger_s3_continuity_2} we get
\begin{equation}
\label{eq:trigger_s3_continuity}
q_{k_{1}}(p, \lambda, d) \le L_{k_{1}}^{-2d}
    \quad \text{for $p \geq p_0(\lambda, d)$},
\end{equation}
concluding the trigger step. This implies that there is no percolation
for these values of $p, \lambda$ and $d$.
In fact, from \eqref{eq:trigger_s3_continuity} and Lemma~\ref{lema:step_s2}
we get we get $q_{k}(p, \lambda, d) \le L_{k}^{-2d}$ whenever $k \geq k_1$
and since $\theta(p,\lambda,d) \leq \lim_{k\to\infty} q_k(p,\lambda,d) =0$
we have
\begin{equation}
\label{eq:continuity_at_1_inside}
\theta(p, \lambda, d) = 0 \quad \text{for $p \geq p_0$}, 
\end{equation}
that implies~\eqref{eq:continuity_at_1}. In other words, this means that
for any $\lambda < \pcb(d)$ the critical curve restricted to the interval 
$[p_0(\lambda, d), 1]$ must be above the horizontal segment of height
$\lambda$.

On the other hand, a consequence
of~\cite[Theorem~2.4 $(iii)$]{beaton2019alignment} is that for
$\lambda > \pcb(d)$ the critical curve must be below $\lambda$ on some
interval $[\tilde{p}_0(\lambda, d), 1]$, since the curve giving the upper bound 
near $p=1$ is continuous, recall Figure~\ref{fig:phase_transition_diagram}.
The continuity of $\lambda_c(\cdot, d)$ at $p=1$ follows.
\end{proof}


\end{document}